\documentclass[12pt,a4paper]{amsart} 
\usepackage{amsmath}
\usepackage{amsthm}
\usepackage{amssymb}
\usepackage{ot1patch}
\usepackage{mathrsfs}

\newtheorem{thm}{Theorem}[section]
\newtheorem{lem}[thm]{Lemma}
\newtheorem{cor}[thm]{Corollary}
\newtheorem{prop}[thm]{Proposition}

\theoremstyle{remark}
\newtheorem{rem}[thm]{Remark}
\newtheorem*{rem*}{Remark}

\theoremstyle{definition}
\newtheorem{dfn}[thm]{Definition}
\newtheorem{ex}[thm]{Example}

\numberwithin{equation}{section}

\newcommand{\Rz}{\mathbb{R}}

\begin{document}

\title{On definable multifunctions and \L ojasiewicz inequalities}

\author{Maciej P. Denkowski \and Paulina Pe\l szy\'nska
}\address{Jagiellonian University, Faculty of Mathematics and Computer Science, Institute of Mathematics, \L ojasiewicza 6, 30-348 Krak\'ow, Poland}\email{maciej.denkowski@uj.edu.pl, paulina.pelszynska@uj.edu.pl}\date{February 8th 2014, \textit{revised:} June 7th 2014, January 30th 2017, June 30th 2017}
\keywords{Tame geometry, multifunctions, \L ojasiewicz inequality, subdifferentials}
\subjclass{32B20, 49J52, 26D10}

\begin{abstract} 
We investigate several possibilities of obtaining a \L ojasie\-wicz inequality for definable multifunctions and give some examples of applications thereof. In particular, we prove that the Hausdorff distance and its extension to closed sets is definable when composed with definable multifunctions. This allows us to obtain \L ojasiewicz-type inequalities for definable multifunctions obtained from Clarke's subgradient or the tangent cone. The paper ends with a \L ojasiewicz-type subgradient inequality in the spirit of Bolte-Daniilidis-Lewis-Shiota or Ph\d{a}m.
\end{abstract}

\maketitle

\section{Introduction}
Throughout this paper \textit{definable} means \textit{definable in some o-minimal structure} 
For a concise presentation of o-minimal structures see e.g. \cite{C} and \cite{vdDM}. We just recall the basic definition:
\begin{dfn}
An \textit{o-minimal structure} (on the field $({\Rz}, +,\cdot)$) is a collection
$\mathcal{S} = \{\mathcal{S}_n\}_{n\in\mathbb{N}}$, where each $\mathcal{S}_n$ is a family of subsets of ${\Rz}^n$ satisfying
the following axioms:
\begin{enumerate}
 \item $\mathcal{S}_n$ contains all the algebraic subsets of ${\Rz}^n$;
\item $\mathcal{S}_n$ is a Boolean algebra (\footnote{Recall that a family $\mathcal{S}$ of sets, subsets of ${\Rz}^n$ in our case, is a {\it
Boolean algebra}, if $\varnothing\in\mathcal{S}$ and for every $A,B\in\mathcal{S}$, there is $A\cap B, A\cup B, {\Rz}^n\setminus
A\in\mathcal{S}$.}) of the powerset of ${\Rz}^n$; \item If $A\in\mathcal{S}_m$, $B\in\mathcal{S}_n$, then $A\times B\in\mathcal{S}_{m+n}$; \item
If $\pi\colon{\Rz}^n\times{\Rz}\to{\Rz}^n$ is the natural projection and $A\in\mathcal{S}_{n+1}$, then $\pi(A)\in\mathcal{S}_n$.
 \item $\mathcal{S}_1$ consists exactly all the finite unions of points and intervals of any type.
\end{enumerate}
The elements of $\mathcal{S}_n$ are called \textit{definable} subsets of ${\Rz}^n$.\\ A function $f\colon A\to {\Rz}^n$, where $A\subset{\Rz}^m$, is called \textit{definable} if its graph, denoted $\Gamma_f$, belongs to $\mathcal{S}_{m+n}$ (then $A\in\mathcal{S}_m$).
\end{dfn}

Let $\mathscr{P}({\Rz}^n)$ denote the family of all subsets of ${\Rz}^n$.  
\begin{dfn} A \textit{definable multifunction} $F\colon {\Rz}^m\to\mathscr{P}({\Rz}^n)$ is a definable subset $E\subset {\Rz}^m\times{\Rz}^n$ such that for each $x\in{\Rz}^m$, $F(x)$ coincides with the \textit{section} $E_x=\{y\in{\Rz}^n\mid (x,y)\in E\}$. We call this set $E$ \textit{the graph of} $F$ and denote it by $\Gamma_F$. Finally, let $\mathrm{dom} F:=\{x\in{\Rz}^m\mid F(x)\neq\varnothing\}$ be the \textit{domain} of $F$ (it is definable as the projection of $\Gamma_F$).
\end{dfn}
\begin{ex}\label{ex}
 A motivating example of such a multifunction can be found in \cite{D}. Namely, given a definable nonempty, closed set $M\subset{\Rz}^m$ we consider for $x\in{\Rz}^n$ the multifunction $m(x)=\{y\in M\mid ||x-y||=d(x,M)\}$ where $d(x,M)$ denotes the Euclidean distance of $x$ to $M$. Note that $m$ has compact values (meaning $m(x)$ are compact). Since $m$ conveys some information about the singularities of $M$, it seems interesting to investigate the possibility of obtaining some \L ojasiewicz-type inequality for this multifunction, as it could give more information about the set $M$.
\end{ex}

In recent years there has been a growing interest in definable multifunctions from optimization specialists (cf. e.g. \cite{DsP}, \cite{BDLewS1}, \cite{BDLewS2} or \cite{DIL}), but apparently there is no much response from people who work in real geometry. The intention of this paper is to present some basic results concerning definable multifunctions. As the \L ojasiewicz inequality (or rather --- inequalities, for there are several types of it) is a major tool of real geometry with important applications (cf. \cite{L2}, \cite{BDLM}), we propose to invsetigate some possible generalizations of it for definable multifunctions. This leads in a natural way to another type of considerations, namely to considering continuity properties of definable multifunctions in the sense of the Kuratowski convergence (cf. \cite{DsP}, \cite{DD}) and the study of the metric that gives this convergence. From this point of view, what can be seen as the central result of the present paper is presented in Theorem \ref{Kuratowski}. Actually, from it we obtain the best versions of \L ojasiewicz-type inequalities. This geometric discussion ends with Theorem \ref{subgradient} giving a \L ojasiewicz-type inequality for the Clarke subgradient in the spirit of \cite{BDLew}, \cite{BDLewS2} (Corollary 16) and \cite{Ph}. 

\medskip
We recall (cf. \cite{L}) that two closed sets $X,Y\subset {\Rz}^m$ are \textit{regularly separated} at $a\in X\cap Y$ iff $d(x,X)+d(x,Y)\geq Cd(x,X\cap Y)^\ell$ in a neighbourhood of $a$, for some constants $C,\ell>0$. This is equivalent to $d(x,Y)\geq Cd(x,X\cap Y)^\ell$ for $x\in X$ in a neighbourhood of $a$ (see e.g. \cite{DK} Lemma 1.1). If $f\colon X\to {\Rz}^n$ is a continuous function, then we say that it satisfies the \L ojasiewicz inequality at $a\in X$ iff $||f(x)-f(a)||\geq Cd(x,f^{-1}(f(a)))^\ell$ in a neighbourhood of $a$, for some $C,\ell>0$. These inequalities are presented in \cite{L}. The point is sets and functions  that are subanalytic or definable in \textit{polynomially bounded} o-minimal structures satisfy this kind of inequalities. Of course, it does not matter which of the usual metrics on ${\Rz}^m$ we consider. Therefore, we will assume that ${\Rz}^m\times{\Rz}^n$ is endowed with the sum of the usual Euclidean norms in ${\Rz}^m$, ${\Rz}^n$. 

By the way, note that if one of the inequalities is satisfied with an exponent $\ell$, it is still satisfied with an exponent $\ell'>\ell$. The greatest lower bound of the possible exponents is called the regular separation or the \L ojasiewicz exponent, according to the case.

 We note the following easy result which we shall refer to later on. 

\begin{prop}\label{podst}
Let $X\subset{\Rz}^m$ be a closed set, $f\colon X\to {\Rz}^n$ a continuous function, $a\in X$ and let $C,\ell>0$. Then \begin{enumerate}
\item If $d(z,X\times\{f(a)\}))\geq Cd(z,(X\times\{f(a)\})\cap \Gamma_f)^\ell$ for $z\in \Gamma_f$ in a neighbourhood of $(a,f(a))$, then $||f(x)-f(a)||\geq Cd(x,f^{-1}(f(a)))^\ell$ in a neighbourhood of $a$;

\item If $||f(x)-f(a)||\geq Cd(x,f^{-1}(f(a)))^\ell$ in a neighbourhood of $a$, then $d(z,X\times\{f(a)\}))\geq \left[\frac{C}{C+1}d(z,(X\times\{f(a)\})\cap \Gamma_f)\right]^{\max\{\ell,1\}}$ for $z\in \Gamma_f$ in a neighbourhood of $(a,f(a))$.
\end{enumerate}
\end{prop}
\begin{proof}
 We compute for $z=(x,f(x))$, $$d(z,X\times\{f(a)\}))=\inf\{||x-x'||+||f(x)-f(a)||\colon x'\in X\}=||f(x)-f(a)||.$$ On the other hand, 
\begin{align*}
d(z,(X\times\{f(a)\})\cap \Gamma_f)&=d(z,f^{-1}(f(a))\times\{f(a)\})=\\
&=\inf\{||x-x'||+||f(x)-f(a)||\colon x'\in f^{-1}(f(a))\}\geq\\
&\geq d(x,f^{-1}(f(a))).
\end{align*}

From this we obtain (1). In order to prove (2), we write for $x$ in a neighbourhood of $a$ such that $||f(x)-f(a)||<1$,
\begin{align*}
d(z,(X\times\{f(a)\})\cap \Gamma_f)&=d(x,f^{-1}(f(a)))+||f(x)-f(a)||\leq\\
&\leq \frac{1}{C}||f(x)-f(a)||^{\frac{1}{\ell}}+||f(x)-f(a)||\leq \\
&\leq \left(\frac{1}{C}+1\right)||f(x)-f(a)||^{\min\{1,1/\ell\}}
\end{align*}
and we are done.\end{proof}
\begin{rem}\label{obs}
The regular separation exponent $\ell$ is always at least 1, provided $a\notin \mathrm{int}(X\cap Y)$, which follows from the inequality $d(x,X)+d(x,Y)\leq 2d(x,X\cap Y)$. On the other hand, the \L ojasiewicz exponent of a function may be smaller than 1.
\end{rem}
\begin{ex}
The function $\sqrt{|x|}$, $x\in{\Rz}$ satisfies the \L ojasiewicz inequality at zero with best exponent $1/2$, while the best exponent for the separation of the graph is $1$.
\end{ex}

Both the regular separation property and the \L ojasiewicz inequality follow from another type of \L ojasiewicz inequality given for two real-valued continuous functions $f,g$ with $f^{-1}(0)\subset g^{-1}(0)$ as $|f(x)|\geq C|g(x)|^q$. 

If the structure is not polynomially bounded, we have the following counterpart of the \L ojasiewicz inequality:
\begin{thm}[\cite{vdDM}]\label{vdDM}
 If $f,g\colon A\to{\Rz}$ are continuous definable functions with $f^{-1}(0)\subset g^{-1}(0)$ and $A\subset{\Rz}^n$ is compact, then there exists an odd $\mathscr{C}^p$ definable, strictly increasing bijection $\phi\colon {\Rz}\to{\Rz}$ which is $p$-flat at zero(\footnote{i.e. $\phi^{(k)}(0)=0, k=0,\dots, p$.}), such that $|\phi(g(t))|\leq |f(t)|$ on $A$.
\end{thm}
We will call such a function $\phi$ a \textit{flattening function}. This plays the role of the general \L ojasiewicz inequality which is stated with $\phi(t)=C\mathrm{sgn}(t)|t|^\ell$. 

\section{General \L ojasiewicz inequalities}

Although usually the \L ojasiewicz inequality is stated for continuous functions (and continuity is used in the standard proofs), a slightly more general result is true, as we will see below. The idea is to use the same kind of argument as in Proposition \ref{podst}.

Let $f\colon \Omega\to{\Rz}^n$ be a subanalytic function on a set $\Omega\subset{\Rz}^m$, \textsl{not necessarily continuous}. 

We want to show that such a function also satisfies a \L ojasiewicz inequality at a zero, i.e. its growth near a zero is estimated from below by the distance to the zero set. 

Let us first agree on what should be considered a zero of $f$. In view of the example of $f(x)=|x|+\chi_{\{0\}}$ (\footnote{Where $\chi_{\{0\}}$ is the characteristic function of $\{0\}\subset{\Rz}$, i.e. $\chi_{\{0\}}(0)=1$ and $\chi_{\{0\}}(x)=0$, for $x\neq 0$.}) it is natural to  consider the generalized set of zeroes as the set --- which we may call \textit{general set of zeroes} ---
$$
f^{-1}(0)^g:=\{x\in \Omega \mid (x,0)\in\overline{\Gamma_f}\}.
$$
If $\Omega$ is subanalytic, or an open set, then $f^{-1}(0)^g$ is subanalytic (either in ${\Rz}^m$, or in $\Omega$).

\begin{thm}
If $f$ is as above, then for any point $a$ in the general set of zeroes of $f$ there is a neighbourhood $U$ and positive constants $C,\ell$ such that
$$
||f(x)||\geq C d(x, f^{-1}(0)^g)^\ell,\quad x\in U\cap\Omega.
$$
\end{thm}
\begin{rem}
It is natural to assume some kind of regularity of $\Omega$, or, instead, when using the regular separation of $\overline{\Gamma_f}$ with ${\Rz}^m\times\{0\}$, to extend the definition of the general set of zeros to $f^{-1}(0)^g:=\{x\in{\Rz}^n\mid (x,0)\in \overline{\Gamma_f}\}$. Either way, the following proof works.
\end{rem}
\begin{proof}[Proof of the Theorem] 
In ${\Rz}^m\times{\Rz}^n$ we consider the product norm $||(x,y)||=||x||+||y||$ and in either of the spaces we may use the $\ell_1$ norm, too.

Let $(a,0)\in f^{-1}(0)^g$. Then there is a ball $B((a,0),\varepsilon)$ and constants $C,\ell>0$ such that 
$$
d((x,y),\overline{\Gamma_f})+d((x,y),\Omega\times\{0\})\geq C d((x,y),\overline{\Gamma_f}\cap (\Omega\times\{0\})^\ell,\> (x,y)\in B((a,0),\varepsilon).
$$
Therefore, if we take $(x,y)\in\overline{\Gamma_f}\cap B((a,0),\varepsilon)$ we have on the one hand, 
$$
d((x,y),\Omega\times\{0\})=\inf\{||x-x'||+||y||\colon x'\in \Omega\}$$
i.e. $d((x,y),\Omega\times\{0\})=||y||$, while on the other,
\begin{align*}
d((x,y),\overline{\Gamma_f}\cap (\Omega\times\{0\}))&=\inf\{||x-x'||+||y||\colon x'\in f^{-1}(0)^g\}\geq \\
&\geq \inf\{||x-x'||\colon x'\in f^{-1}(0)^g\}=\\
&=d(x,f^{-1}(0)^g).
\end{align*}
Thus, $||y||\geq C d(x,f^{-1}(0)^g)^\ell$ whenever $(x,y)\in\overline{\Gamma_f}\cap B((a,0),\varepsilon)$. 

Observe that we have $\ell\geq 1$ (Remark \ref{obs}) and we may always assume that $C\in (0,1)$. Moreover, let us assume that $\varepsilon<1$.  

Now, for any $x\in B_m(a,\varepsilon)\cap\Omega $ we have either $f(x)\in B_n(0,\varepsilon)$ and so 
$$
||f(x)||\geq C d(x,f^{-1}(0)^g)^\ell,
$$
or $||f(x)||\geq\varepsilon$. In the latter case, we obtain
\begin{align*}
||f(x)||&\geq\varepsilon>||x-a||\geq\\
&\geq d(x,f^{-1}(0)^g)\geq\\
&\geq Cd(x,f^{-1}(0)^g)^\ell,
\end{align*}
by the choice of $C$ and $\ell$. 
\end{proof}

In an o-minimal structure, if $X,Y\subset{\Rz}^n$ are closed, definable and nonempty, then by applying Theorem \ref{vdDM} to the definable, continuous functions $f(x)=d(x,Y)$ and $g(x)=d(x,X\cap Y)$, we obtain the inequality
$$
d(x,Y)\geq \phi(d(x,X\cap Y))\leqno{(*)}
$$
in a neighbourhood of $a\in X\cap Y$ in $X$ with some flattening bijection $\phi$. 

Therefore, we may more or less repeat the last proof in order to obtain the following:
\begin{thm}
If $f\colon \Omega\to{\Rz}^m$ is a definable function, then for any point $a$ in the general set of zeroes of $f$ there is a neighbourhood $U$ and a continuous, odd, bijection of class $\mathscr{C}^p$ $\phi\colon {\Rz}\to{\Rz}$ that is $p$-flat at zero and such that
$$
||f(x)||\geq  \phi(d(x, f^{-1}(0)^g)),\quad x\in U\cap\Omega.
$$
\end{thm}
\begin{proof}
Using $(*)$ we obtain for $(x,y)\in\overline{\Gamma_f}\cap B((a,0),\varepsilon)$,
$$
||y||\geq  \phi(d((x,y), \overline{\Gamma_f}\cap(\Omega\times\{0\})).
$$

Of course, the distances are computed as in the previous proof. The monotonicity of $\phi$ allows us to adapt this proof to the present case. A slight modification is needed at the end, as follows.

We know that $\phi$ is $p$-flat at zero and so 
in particular, for $t>0$ sufficiently close to zero, we necessarily have $\phi(t)\leq t$. Assume that $\varepsilon$ is chosen so that this inequality holds.

Then, if for some $x\in B_n(a,\varepsilon)$ we have $||f(x)||\geq \varepsilon$,  this leads to
$$
||f(x)||\geq\varepsilon>||x-a||\geq d(x,f^{-1}(0)^g)\geq \phi(d(x,f^{-1}(0)^g))
$$
and we are done.
\end{proof}

All this allows us to disregard the problem of continuity in what will follow. 

\section{Preliminaries}
In order to obtain a \L ojasiewicz-type inequality for a definable multifunction $F$ in a polynomially bounded o-minimal structure one method consists in using the regular separation of the graph $\Gamma_F$ with the domain as in Proposition \ref{podst}. However, first we should define a notion of pre-image of $F$. There are several natural possibilites.

Let $a\in\mathrm{dom} F$, we consider \begin{itemize}
\item $F^{-1}(F(a))=\{x\in{\Rz}^m\mid F(x)=F(a)\}$ \textit{the (strong) pre-image};
\item $F^*(F(a))=\{x\in \mathrm{dom} F\mid F(x)\subset F(a)\}$ \textit{the lower pre-image};
\item $F_*(F(a))=\{x\in{\Rz}^m\mid F(x)\supset F(a)\}$ \textit{the upper pre-image};
\item $F^\#(F(a))=\{x\in{\Rz}^m\mid F(x)\cap F(a)\neq\varnothing\}$ \textit{the weak pre-image}. 
\end{itemize}

Finally, we may consider a \textit{point pre-image} defined for a point $y\in F(a)$ as the section $(\Gamma_F)_y:=\{x\in {\Rz}^m\mid y\in F(x)\}$. Obviously, $$F^\#(F(a))=\bigcup_{y\in F(a)}(\Gamma_F)_y.$$

\begin{prop}
All the sets defined above are definable.
\end{prop}
\begin{proof}
Let $\pi(x,y)=x$ for $(x,y)\in{\Rz}^m\times {\Rz}^n$. It is clear that  
the point pre-image is definable, as $(\Gamma_F)_y=\pi(\Gamma_F\cap ({\Rz}^m\times\{y\}))$.

On the other hand, \begin{align*}F^\#(F(a))&=\{x\in{\Rz}^m\mid \exists y\in F(a)\colon y\in F(x)\}=\\
&=\pi(\{(x,y)\in{\Rz}^m\times F(a)\mid (x,y)\in\Gamma_F\})=\\
&=\pi(({\Rz}^m\times F(a))\cap\Gamma_F)\end{align*}
whence the definability of the weak pre-image.

Next, note that $F^{-1}(F(a))=F^*(F(a))\cap F_*(F(a))$, thus we concentrate on the lower an upper pre-images. Now, 
\begin{align*}
\mathrm{dom} F\setminus F^*(F(a))&=\{x\in\mathrm{dom} F\mid \exists y
\in F(x)\setminus F(a)\}=\\
&=\pi(\{(x,y)\in\mathrm{dom} F\times {\Rz}^n\mid y\in F(x)\setminus F(a)\})=\\
&=\pi(\Gamma_F\setminus ({\Rz}^m\times F(a))),
\end{align*}
where $\pi(x,y)=x$, which accounts for the definability of the lower pre-image.

Finally, define $G(x)={\Rz}^n\setminus F(x)$. Then $\Gamma_G={\Rz}^{m+n}\setminus\Gamma_F$ and so $G$ is a definable multifunction. It remains to observe that $F_*(F(a))=G^*(G(a))$ and we are done.
\end{proof}
\begin{rem}
Observe that all the notions introduced above can be of some interest, e.g. in the case of Example \ref{ex}.
\end{rem}
One should be careful in the subanalytic case, since all subanalytic sets do not form an o-minimal structure.
\begin{ex}
The set $$\bigcup_{\nu=1}^{+\infty} \{1/\nu\}\times [\nu,+\infty)$$ is a subanalytic subset of ${\Rz}^2$ and so may be seen as the graph of a subanalytic multifunction $F\colon {\Rz}\to\mathscr{P}(\Rz)$. However, $$F^*(F(1))=\{x\in\mathrm{dom} F\mid F(x)\subset [1,+\infty)\}=\bigcup_{\nu=1}^{+\infty} \{1/\nu\}$$ is not subanalytic.

Also, $F^\#(F(1))=F^*(F(1))$ in this case.

Note that even requiring $\mathrm{dom} F$ to be subanalytic does not help much, since it suffices to slightly modify the example above. Namely let $G$ be the subanalytic multifunction whose graph is 
$$
\{(x,1/x)\mid x\in (0,+\infty)\}\cup(\{1\}\times\mathbb{N}).
$$
Then $\mathrm{dom} G=(0,+\infty)$, yet again $G^*(G(1))=G^\#(G(1))=F^*(F(1))$.

By the way, observe that both multifunctions have closed graphs.
\end{ex}

We will assume from now on that $\Gamma_F$ is \textsl{closed}. This yields the upper semi-continuity of $F$ in a sense explained below. 

 Let $\mathcal{F}_n$ denote the collection of all the closed subsets of ${\Rz}^n$ endowed with the topology of the Kuratowski convergence (see e.g. \cite{TW}, \cite{DP}). 

\begin{dfn}
We say that a sequence $(A_\nu)\subset\mathcal{F}_n$ \textit{converges (in the sense of Kuratowski)} to $A$ (which in this case is necessarily a closed set, maybe empty), if any point $x\in A$ is the limit of a sequence of points $x_\nu\in A_\nu$ and for each compact set $K$ such that $K\cap A=\varnothing$, condition $K\cap A_\nu=\varnothing$ holds for almost all indices. We write then $A_\nu\stackrel{K}{\longrightarrow} A$.
\end{dfn}

We recall that the Kuratowski convergence is particulary interesting in the case of multifunctions. 

\begin{dfn}
For an accumulation point $a$ of $\mathrm{dom} F$ we define the Kuratowski \textit{upper limit}:
$$
y\in\limsup_{\mathrm{dom F}\ni x\to a}F(x)\ \Leftrightarrow\ \forall U\ni y, \forall V\ni a, \exists x\in V\setminus \{a\}\colon U\cap F(x)\neq\varnothing,
$$
where $U,V$ are open sets. Similarly, we define \textit{the lower limit}:
$$
y\in\liminf_{\mathrm{dom} F\ni x\to a} F(x)\ \Leftrightarrow\ \forall U\ni y, \exists V\ni a\colon \forall V\cap\mathrm{dom} F\setminus\{a\}, U\cap F(x)\neq\varnothing.
$$
We obviously have $\liminf_{\mathrm{dom} F\ni x\to a} F(x)\subset \limsup_{\mathrm{dom} F\ni x\to a} F(x)$ and if the converse inclusion holds, we speak of convergence. The limit set is equal to $F(a)$ iff $\limsup_{\mathrm{dom} F\ni x\to a} F(x)\subset F(a)\subset \liminf_{\mathrm{dom} F\ni x\to a} F(x)$ and we write $F(x)\stackrel{K}{\longrightarrow} F(a)$. In that case we call $a$ a \textit{continuity point of} $F$. 
\end{dfn}

As observed before, if $\Gamma_F$ is closed, then $F$ is upper semi-continuous in the sense that for each accumulation point $a$ of $\mathrm{dom} F$ we have the inclusion (see e.g. \cite{DD}) $\limsup_{\mathrm{dom F}\ni x\to a}F(x)\subset F(a)$.

A more detailed study of the Kuratowski limits in tame geometry can be found in \cite{DD}. Here, we note just one result whose simple proof is left to the reader:
\begin{prop}\label{obs2}
 Let $a$ be an accumulation point of $\mathrm{dom} F$, where $F$ is a multifunction. Then \begin{itemize}
 \item $y\in\limsup_{\mathrm{dom F}\ni x\to a}F(x)$ iff there is a sequence $\mathrm{dom} F\setminus\{a\}\ni x_\nu\to a$ and points $F(x_\nu)\ni y_\nu\to y$;

 \item $y\in\liminf_{\mathrm{dom} F\ni x\to a} F(x)$ iff for any sequence $\mathrm{dom} F\setminus\{a\}\ni x_\nu\to a$ we can find points $F(x_\nu)\ni y_\nu\to y$.
 \end{itemize}
\end{prop}

\section{Regular separation}
\begin{thm}
If $F$ is a definable multifunction in a polynomially bounded o-minimal structure and with closed graph, then for any $a\in\mathrm{dom} F$ and $y\in F(a)$ we can find a neighbourhood $U$ of $(a,y)\in \Gamma_F$ and constants $C,\ell>0$ such that 
$$
||v-y||\geq Cd(x,(\Gamma_F)_y)^\ell,\quad (x,v)\in \Gamma_F\cap U.
$$
\end{thm}
\begin{proof}
Fix $a\in\mathrm{dom} F$ and $y\in F(a)$. From the regular separation of $\Gamma_F$ and ${\Rz}^m\times\{y\}$ at their common point $(a,y)$ we get in a neighbourhood $U$ of this point
$$
d((x,v),{\Rz}^m\times\{y\})\geq Cd((x,v),\Gamma_F\cap ({\Rz}^m\times\{y\}))^\ell,\quad (x,v)\in \Gamma_F\cap U.
$$
We compute first $$
d((x,v),{\Rz}^m\times\{y\})=\inf\{||x-x'||+||v-y||\colon x'\in{\Rz}^m\}=||v-y||.
$$
Next, \begin{align*}
d((x,v),\Gamma_F\cap ({\Rz}^m\times\{y\}))&=\inf\{||x-x'||+||v-y||\colon x'\in (\Gamma_F)_y\}\geq\\
&\geq \max\{d(x,(\Gamma_F)_y),||v-y||\},
\end{align*}
whence we get the inequality sought for.
\end{proof}

\begin{thm}\label{odleglosc}
In the setting considered, for any compact set $K\subset {\Rz}^{m+n}$ there are constants $C,\ell>0$ such that
$$
d(y,F(x))\geq C d((x,y),\Gamma_F)^\ell,\quad (x,y)\in K, x\in\mathrm{dom} F.
$$
Moreover, if do not assume the structure to be polynomially bounded, we can find a $\mathcal{C}^p$ definable, strictly increasing bijection $\phi\colon {\Rz}\to {\Rz}$ that is $p$-flat at zero and such that
$$
d(y,F(x))\geq \phi(d((x,y),\Gamma_F)),\quad (x,y)\in K, x\in\mathrm{dom} F.
$$
\end{thm}
\begin{proof}
Consider the function $$\delta\colon {\Rz}^m\times{\Rz}^n\ni (x,y)\mapsto d(y,F(x))\in{\Rz}.$$ We check as in \cite{LW} that $\delta$ is definable. Note that $\delta(x,y)=0$ iff $y\in F(x)$. This and the classical \L ojasiewicz inequality allows us to state the first part of the theorem. The second follows from Theorem \ref{vdDM}.
\end{proof}

\section{Multifunctions with compact values and the Hausdorff distance}
Assume that $F$ has compact values $F(x)$ and recall \textit{the Hausdorff distance}:
for nonempty compact subsets of $\mathbb{R}^n$ we put
$$
\mathrm{dist}_H(K,L)=\max\{\max_{x\in L} d(x,K),\max_{y\in K} d(y,L)\}.
$$
We will need some preliminary results.
\begin{lem}\label{max}
Let $f,g\colon A\to {\Rz}$ be definable functions on $A\subset{\Rz}^m$. Then $\varphi(x):=\max\{f(x),g(x)\}$ is definable, too.
\end{lem}
\begin{proof}
The sets $A_+:=\{x\in A\mid f(x)\geq g(x)\}$ and $A_-:=\{x\in A\mid f(x)<g(x)\}$ are both definable and their union is $A$. Then $\varphi=f|_{A_+}\cup g|_{A_-}$.
\end{proof}
\begin{lem}\label{skladowa}
Let $F\colon {\Rz}^m\to \mathscr{P}({\Rz}^n)$ be a definable multifunction and $f\colon {\Rz}^k\times{\Rz}^n\to {\Rz}$ a definable function. Then the function
$$
\psi\colon {\Rz}^m\times \mathrm{dom} F\ni (x,x')\mapsto \sup_{y\in F(x')}f(x,y)\in{\Rz}\cup\{\infty\}
$$
is definable.
\end{lem}
\begin{proof}
First we note that 
$$
\psi(x,x')=+\infty\ \Leftrightarrow\ \forall M>0,\exists y\in F(x')\colon f(x,y)\geq M.
$$
In view of the fact that the function $f$ is definable and that $y\in F(x')$ iff $(x',y)\in \Gamma_F$ we conclude that the set of points at which $\psi$ is infinite is definable (as it is described by a first order formula). We may thus assume that $\psi$ takes only finite values.

Now it remains to observe that the graph of $\psi$ consists of those points $(x,x',t)$ for which\begin{itemize}
\item $f(x,y)\leq t$ for all $y$ such that $(x',y)\in\Gamma_f$,
\item $\forall \varepsilon\in (0,1), \exists y\in F(x')\colon f(x,y)>t-\varepsilon$.
\end{itemize}
This description by first order formul{\ae} accounts for the definability of $\Gamma_\psi$.
\end{proof}
\begin{prop}\label{Delta}
If $F\colon {\Rz}^m\to\mathscr{P}({\Rz}^n)$, $G\colon {\Rz}^k\to\mathscr{P}({\Rz}^n)$ are definable multifunctions with compact values, then the function
$$
\Delta_{F,G}\colon \mathrm{dom} F\times\mathrm{dom} G\ni (x,x')\mapsto \max_{y\in G(x')} d(y,F(x))\in{\Rz}
$$
is definable.
\end{prop}
\begin{proof}
Let $\delta(x,y)=\mathrm{dist}(y,F(x))$ as in the proof of Theorem \ref{odleglosc}. Then $\Delta_{F,G}(x,x')=\max_{y\in G(x')}\delta(x,y)$ and the result follows from the preceding lemma.
\end{proof}

Summing up, we obtain:
\begin{thm}\label{Hausdorff}
If $F\colon {\Rz}^m\to\mathscr{P}({\Rz}^n)$, $G\colon {\Rz}^k\to\mathscr{P}({\Rz}^n)$ are definable multifunctions with compact values, then the function
$$
d_H(F,G)\colon \mathrm{dom} F\times\mathrm{dom} G\ni (x,x')\mapsto \mathrm{dist}_H(F(x),G(x'))\in{\Rz}
$$
is definable.
\end{thm}
\begin{proof}
This follows directly from the Lemma \ref{max} and Proposition \ref{Delta} together with the formula for the Hausdorff distance.
\end{proof}
This theorem yields a \L ojasiewicz type inequality for the strong pre-image, since $d(x,F^{-1}(F(a)))=0$ implies $\mathrm{dist}_H(F(x),F(a))=0$. 
\begin{cor}\label{cor1}
Let $F$ be a definable multifunction with compact values. Assume the o-minimal structure to be polynomially bounded. Then for any $a\in\mathrm{dom} F$ there is a neighbourhood $U\ni a$ and constants $C,\ell>0$ such that 
$$
\mathrm{dist}_H(F(x),F(a))\geq Cd(x,F^{-1}(F(a)))^\ell,\quad x\in U\cap\mathrm{dom} F.
$$
\end{cor}
In a general o-minimal structure we have, of course, the following version:
\begin{cor}\label{cor2}
Let $F$ be a definable multifunction with nonempty, compact values. Then for any $a\in\mathrm{dom} F$ there is a neighbourhood $U\ni a$ and a $\mathcal{C}^p$ definable, strictly increasing bijection $\phi\colon {\Rz}\to{\Rz}$ $p$-flat at zero and such that 
$$
\mathrm{dist}_H(F(x),F(a))\geq \phi(d(x,F^{-1}(F(a)))),\quad x\in U\cap\mathrm{dom} F.
$$
\end{cor}
Actually, these two corollaries hold true for any type of pre-image:
\begin{cor}\label{cor3} Both Corollaries \ref{cor1} and \ref{cor2} hold for the upper, lower, weak and point pre-images.
\end{cor}
\begin{proof}
It is a straightforward consequence of the fact that $$F^{-1}(F(a))= F^*(F(a))\cap F_*(F(a))\cap F^\#(F(a))$$ as well as $(\Gamma_F)_y\subset F^\#(F(a))$ for $y\in F(a)$, together with the simple facts that
$A\subset B$ implies $d(x,A)\geq\mathrm d(x,B)$ and the functions we consider, namely $[0,+\infty)\ni t\mapsto Ct^\ell$ (Corollary \ref{cor1}) and $\phi$ (Corollary \ref{cor2}) are increasing.
\end{proof}

\section{Closed multifunctions}\label{Kur}

In this section we will generalize the results from the previous one to the case of closed (i.e. with closed values), definable multifunctions.

We will need two simple lemmata.
\begin{lem}
If $F\colon {\Rz}^m\to\mathscr{P}({\Rz}^n)$ is a definable multifunction and $h\colon {\Rz}^n\to{\Rz}^p$ a definable function, then $G\colon \mathrm{dom} F\ni x\mapsto h(F(x))\in\mathscr{P}({\Rz}^p)$ is a definable multifunction.
\end{lem}
\begin{proof}
It follows readily from the description of the graph $\Gamma_G$.
\end{proof}
\begin{lem}
If $F,G\colon {\Rz}^m\to\mathscr{P}({\Rz}^n)$ are two definable multifunctions and $*$ denotes either $\cap$, $\cup$, or $\setminus$, then $H(x)=F(x)*G(x)$ is a definable multifunction.
\end{lem}
\begin{proof}
Clearly, $\Gamma_H=\Gamma_F*\Gamma_G$.
\end{proof}

\begin{thm}\label{Kuratowski}
There is a metric $\mathrm{dist}_K$ on $\mathcal{F}_n$ extending the Hausdorff metric and such that for any two definable, closed multifunctions $F\colon {\Rz}^m\to\mathscr{P}({\Rz}^n)$, $G\colon {\Rz}^k\to\mathscr{P}({\Rz}^n)$ the induced function
$$
d_K(F,G)\colon {\Rz}^m\times{\Rz}^k\ni (x,x')\mapsto \mathrm{dist}_K(F(x),G(x'))\in{\Rz}
$$
is definable and there is a definable set $W\subset \mathrm{dom} F\times\mathrm{dom} G$ with $\dim W<\dim\mathrm{dom} F+\dim\mathrm{dom} G$ and such that $d_K(F,G)$ is continuous at each point $(x,x')\in \mathrm{dom} F\times\mathrm{dom} G\setminus W$.
\end{thm}
\begin{proof}
 We start with a one point compactification of ${\Rz}^n$. Namely, we identify ${\Rz}^n$ with $R:={\Rz}^n\times\{-1\}\subset{\Rz}^{n+1}$ and we consider the unit sphere $\mathbb{S}^n$ with the stereographic projection $s$ from the north pole $p=(0,\dots, 0,1)$ onto $R$, $$
 s(x)=\left(\frac{-2x_1}{x_{n+1}-1},\dots,\frac{-2x_n}{x_{n+1}-1},-1\right)\in R\ \textrm{for}\  x\in\mathbb{S}^n\setminus\{p\}.
 $$ 
Let $\varrho$ denote the restriction of the Euclidean metric in ${\Rz}^{n+1}$ to the sphere. 

We are thus in a nice semi-algebraic setting and each o-minimal structure contains semi-algebraic sets. Moreover, it is clear that $(\mathbb{S}^n,\varrho)$ corresponds to the standard Alexandrov's one-point compactification of $R$ and $h:=s^{-1}$ is a semi-algebraic  homeomorphism $R\to \mathbb{S}^n\setminus\{p\}$. 

Now $\mathcal{F}_n$ denotes all the closed subsets of $R$. It is easy to check (using for instance Lemma 2 from \cite{TW} and the Kuratowski convergence) that 
 $$
 \mathrm{dist}_K(K,L)=\mathrm{dist}_H(h(K)\cup\{p\}, h(L)\cup\{p\}),\quad K,L\in \mathcal{F}_n
 $$
 defines a metric on $\mathcal{F}_n$ that agrees with the Kuratowski convergence. Note that here $\mathrm{dist}_H$ is an extended version of the Hausdorff metric on $\mathbb{S}^n$. Namely, if only one of the the sets $S,T\subset\mathbb{S}^n$ is empty, then we put $\mathrm{dist}_H(S,T)=\mathrm{diam}\mathbb{S}^n+1$. 
 
Let $F'(x):=h(F(x))\cup\{p\}$ for $x\in\mathrm{dom} F$ and similarly define $G'$. By the lemmata preceding the theorem, $F',G'$ are definable multifunctions. 
 
 In order to end the proof of the definability of $d_K(F,G)$, it suffices to observe that 
 $$
 d_K(F,G)=d_H(F',G')
 $$
 and use Theorem \ref{Hausdorff}.
 
 The second part of the statement follows from the fact that the convergence in $\mathrm{dist}_K$ is precisely the Kuratowski convergence. Moreover, by Theorem 2.11 in \cite{DD}, there is a definable set $Z_F\subset\mathrm{dom} F$ such that $\dim Z_F<\dim\mathrm{dom} F$ and $F$ is continuous apart from $Z_F$. Take an analoguous set $Z_G$ for $G$. Then the set $W=(\mathrm{dom} F\times\mathrm{dom} G)\setminus [(\mathrm{dom} F\setminus Z_F)\times(\mathrm{dom} G\setminus Z_G)]$ is the set sought after.
\end{proof}

\begin{cor}\label{wniosek}
The natural counterparts of Corollaries \ref{cor1}, \ref{cor2} and \ref{cor3} hold in the case considered.
\end{cor}

Before the next corollary we note one useful fact
\begin{lem}[\cite{D2} Lemma 2.1]
Let $E\subset{\Rz}^k_t\times{\Rz}^n_x$ be a closed, nonempty set with continuously varying sections $E_t$ over $F:=\pi(E)$ where $\pi(t,x)=t$. Then the function
$$
\delta(t,x):=d(x,E_t),\quad (t,x)\in F\times{\Rz}^n
$$
is continuous.
\end{lem}
Now we obtain a generalization of Theorem \cite{vdDM}.
\begin{cor}\label{nierownosc}
Assume that $F\colon {\Rz}^m\to\mathscr{P}({\Rz}^n)$ and $G\colon {\Rz}^m\to\mathscr{P}({\Rz}^n)$ are continuous definable multifunctions with $\mathrm{dom} F\cap\mathrm{dom} G\neq\varnothing$ and $A\subset{\Rz}^n$ is a closed set for which
$$
F(x)=A\Rightarrow G(x)=A\vee x\notin\mathrm{dom} G.
$$
Then for any point $x$ for which $F(x)=G(x)=A$ there is a neighbourhood and a flattening function $\phi$ such that
$$
\mathrm{dist}_K(F(x),A)\geq \phi(\mathrm{dist}_K(G(x),A))
$$
in this neighbourhood.
\end{cor}
\begin{proof}
By the preceding Theorem $f(x):=\mathrm{dist}_K(F(x),A)$ and $g(x):=\phi(\mathrm{dist}_K(G(x),A))$ are continuous (cf. it follows easily from the preceding Lemma), definable functions. It is enough to apply Theorem \cite{vdDM} on a compact neighbourhood.
\end{proof}

\begin{rem}
Similar results can be obtained for instance for the composition $G\circ F$ of two definable multifunctions $F\colon {\Rz}^m\to\mathscr{P}({\Rz}^n)$ and $G\colon {\Rz}^n\to\mathscr{P}({\Rz}^p)$ defined by its graph: for $(x,z)\in {\Rz}^m\times{\Rz}^p$,
$$
(x,z)\in \Gamma_{G\circ F}\ \Leftrightarrow \exists y\in{\Rz}^n\colon (x,y)\in\Gamma_F\wedge (y,z)\in\Gamma_G.
$$
Clearly $G\circ F$ is again definable.
\end{rem}

\section{Examples of definable multifunctions}

There are three basic but important examples of definable multifunctions. The first one, $m(x)$ was already introduced in Example \ref{ex}. Two others are presented in the following Theorem \ref{definiowalne}.

Before we state it, let us just recall that for any set $E\subset{\Rz}^m$ and $a\in\overline{E}$ we define the Peano tangent cone of $E$ at $a$ classically as
$$
C_a(E)=\limsup_{t\to 0+}{(1/t)}(E-a).
$$
Of course, it is definable, if $E$ is such (this follows from the description, see e.g. \cite{D}). 

If $a$ is an isolated point of $E$, then clearly the upper limit reduces to $\{0\}$. But the formula is still true for a point $a\notin\overline{E}$. In that case, in view of 
$$
v\in\limsup_{t\to 0^+}(1/t)(E-a)\ \Leftrightarrow\ \forall U\ni v, \forall \delta, \exists t\in (0,\delta)\colon (1/t)(E-a)\cap U\neq\varnothing,
$$
or in other words
$$
v\in \limsup_{t\to 0^+}(1/t)(E-a)\ \Leftrightarrow\ \exists E\ni x_\nu\to a, \exists t_\nu\to 0\colon (x_\nu-a)/t_\nu\to v,
$$
we see that the upper limit is empty. In fact, this means that $a\notin \overline{E}$ implies $(1/t)(E-a)\stackrel{K}{\longrightarrow}\varnothing$.

A cone $V\subset{\Rz}^n$ with vertex at zero is uniquely determined by $V\cap\mathbb{S}^{n-1}$ with the observation that this intersection is empty precisely when $V=\{0\}$.

We shall give some additional properties of the tangent cone in the definable setting. In order to simplify the notation, we will assume $a=0$ for the moment being and put $V=C_0(E)$ regardless of what it the position of the origin with respect to $E$.
\begin{lem}
There always is $$V\cap\mathbb{S}^{n-1}=\limsup_{t\to 0^+} [(1/t)E\cap \mathbb{S}^{n-1}]=[\limsup_{t\to 0^+} (1/t)E]\cap \mathbb{S}^{n-1}.$$
\end{lem}
\begin{proof}
We only need to prove the first equality.

If $0$ is isolated in $E$ or does not belong to $E$, then we obtain the empty set on both sides. Hence, we may assume that the origin is an accumulation point of $E$.

Take any sequence $x_\nu\in\mathbb{S}^{n-1}$ such that $t_\nu x_\nu\in E$ for some $t_\nu\to 0^+$ and $x_\nu\to v$. Then, clearly, $y_\nu:=t_\nu x_\nu\to 0$ and $y_\nu/||y_\nu||\to v$ which means that $v\in V\cap\mathbb{S}^{n-1}$. 

To prove the converse inclusion, for any $v\in V\cap\mathbb{S}^{n-1}$ we can find $E\ni x_\nu\to 0$ such that $y_\nu:=x_\nu/||x_\nu||\to v$ (\footnote{Indeed, formally we have $v=\lim s_\nu z_\nu$ for some $s_\nu>0$ and $E\ni z_\nu\to 0$. But then $s_\nu z_\nu/||s_\nu z_\nu||\to v/||v||=v$.}). Thus for $t_\nu:=||x_\nu||\to 0^+$ we obtain $y_\nu\in (1/t_\nu)E$.
\end{proof}

In the definable or subanalytic case (\footnote{Note that the tangent cone depends only on the germ of the set which implies that the subanalytic case can be dealt with just as the definable one.}), thanks to the Curve Selection Lemma, we can replace the upper limit by the limit itself:
\begin{prop}
If $E$ is definable, then 
$$
V=\lim_{t\to 0^+} (1/t)E\quad \textit{and}\quad V\cap\mathbb{S}^{n-1}=\lim_{t\to 0^+}(1/t)E\cap \mathbb{S}^{n-1}.
$$
Moreover, 
$$V=\{\varphi'(0)\mid \varphi\colon [0,1)\to {\Rz}^n\ \textit{definable}, \varphi(0)=0, \varphi((0,1))\subset E\}.$$
\end{prop}
\begin{proof}
If $V\subset\{0\}$, there is nothing to do. Assume that $0$ is an accumulation point of $E$ and consider
$$
X=\bigcup_{t>0}\{t\}\times (1/t)E=\{(t,x)\mid t>0, tx\in E\}
$$
which obviously is a definable set. We know from \cite{DD} that for the $t$-sections we have
$$
(\overline{X})_0=\limsup_{t\to 0^+} X_t=V.
$$
For any $(0,x)\in (\overline{X})_0$ we can use the Curve Selection Lemma (cf. $X_0=\varnothing$) finding a definable curve that can be written as $\gamma(t)=(t,\eta(t))$ and is $\mathscr{C}^1$. In particular, $t\eta(t)\in E$ and $\lim_{t\to 0^+}\eta(t)=x$. Then for the definable $\mathscr{C}^1$ curve $\varphi(t):=t\eta(t)$ we have $\varphi(t)\in E$ ($t>0$), $\varphi(0)=0$ and $$\varphi'(0)=\lim_{t\to 0^+}\frac{\varphi(t)}{t}=\lim_{t\to 0^+}\eta(t)=x.$$
This ends the proof of the first equality, for whenever we fix a neighbourhood $U\ni x$, we have $\eta(t)\in (1/t)E\cap U$ for all $t$ near zero.

To prove the second one we argue similarly. The set 
$$
X=\bigcup_{t>0}\{t\}\times (1/t)E\cap\mathbb{S}^{n-1}=\{(t,x)\mid t>0, ||x||=1, tx\in E\}
$$
is again definable and $\limsup_{t\to 0^+} X_t=(\overline{X})_0=V\cap \mathbb{S}^{n-1}$. As before $X_0=\varnothing$ and so the Curve Selection Lemma yields a curve $(t,\eta(t))\in \{t\}\times X_t$. Then $||\eta(t)||=1$ and $t\eta(t)\in E$ so that for any neighbourhood $U\ni x$ we have $\eta(t)\in U\cap (1/t)E\cap\mathbb{S}^{n-1}$ for all $t$ sufficiently near zero. This ends the proof.
\end{proof}

As a direct consequence we have that the dilatations of a definable set always form a definable, convergent family whose limit is the tangent cone, also when intersected with the sphere. Note that this can be expressed by writing
$$
\mathrm{dist}_{H}(E\cap\mathbb{S}(r), V\cap\mathbb{S}(r))=o(r)\> \textrm{at zero}.
$$

If $E$ is open and $f\colon E\to {\Rz}$ is a locally Lipschitz function, then by the Rademacher Theorem, it has a well defined (and locally bounded) gradient almost everywhere --- we will denote by $D_f$ this set of differentiability points. Following \cite{Cl} we consider the \textit{generalized gradient at} or \textit{subdifferential} that we prefer to call \textit{subgradient} $\partial f(x)$ at $x\in\mathbb{R}^n$ defined as the convex hull of the set of all possible limits $\lim\nabla f(x_\nu)$ when $x_\nu\to x$. The set $\partial f(x)$ is a nonempty convex compact set. 

As Clarke shows in \cite{Cl}, for a Lipschitz function $f$, we have $\partial f(x)=\{y\}$ iff $x\in D_f$ and $\nabla (f|_{D_f})$ is continuous at $x$; in that case $y=\nabla f(x)$.

\begin{thm}\label{definiowalne}
Let $E\subset{\Rz}^m$ be a definable set and $f\colon E\to{\Rz}^n$ a definable function. Then \begin{enumerate}
\item The multifunction $\tau\colon E\ni x\mapsto C_x(E)\in\mathscr{P}({\Rz}^m)$ is definable;
\item If $E$ is open and $f$ is locally Lipschitz, then $\partial\colon E\ni x\mapsto \partial f(x)\in\mathscr{P}({\Rz}^{mn})$ is definable.
\end{enumerate}
\end{thm}
\begin{proof}
In order to prove (1) we use the definition of the tangent cone.
From this we derive the following description by a first order formula of the graph of $\tau$:
$$
\Gamma_\tau=\{(x,v)\in{\Rz}^m\times{\Rz}^m\mid \forall \varepsilon>0, \forall r>0, \exists t\in (0,\varepsilon)\colon \exists y\in F_{(x,t)}, ||v-y||^2<r^2\}
$$
where $F_{(x,t)}$ denotes the $(x,t)$-section of the definable set
$$
F=\{(x,t,y)\mid x\in E, t>0, ty+x\in E\}.
$$

Property (2) requires the use of the Carath\'eodory Theorem. First, we note that the set 
$$
D_f=\{x\in E\mid \exists d_x f\},
$$
where $d_x f$ is the differential, is definable together with the derivative $D_f\ni x\mapsto d_x f\in{\Rz}^m$ (cf. \cite{C}, \cite{vdDM} as well as the last chapter of \cite{DS}). Let $\Gamma$ denote the closure of its graph. Then the $x$-section of $\Gamma$ is precisely
$$
(\Gamma)_x=\left\{\ell\in{\Rz}^{mn}\mid \exists (x_\nu)\subset D_f\colon \ell=\lim_{\nu\to+\infty} d_{x_\nu} f\right\}.
$$
Finally, by the Carath\'eodory Theorem (\footnote{Any point from a convex hull of a subset of ${\Rz}^n$ is a convex combination of at most $n+1$ points from this set.}), $\Gamma_\partial$ coincides with the set
$$
\left\{(x,y)\mid \exists \ell_0,\dots, \ell_{mn}\in(\Gamma)_x, \exists \lambda_0,\dots,\lambda_{mn}\geq 0\colon \sum_{i=0}^{mn}\lambda_i=1, y=\sum_{i=0}^{mn}\lambda_i\ell_i\right\}
$$
which is definable.
\end{proof}
This Theorem allows us to apply the preceding results to these functions. For any set locally closed $E\subset{\Rz}^n$ containing the origin write $E{[r]}:=E\cap\overline{\mathbb{B}}(0,r)$. In particular we obtain the following general counterpart of \cite{Ch} 2.8.6.
\begin{cor}
If $E\subset {\Rz}^m$ is locally closed and definable in a polynomially bounded o-minimal structure, or subanalytic, $0\in E$, but $E$ is not a cone at zero, then we have
$$
\mathrm{dist}_H(E[r],C_0(E)[r])=cr^\alpha+o(r^\alpha) \quad(r\to 0^+)$$
for some constants $C,\alpha>0$ with $r\mapsto r^\alpha$ definable.
\end{cor}
\begin{proof}
By Theorems \ref{definiowalne} and \ref{Hausdorff} we know that $f\colon r\mapsto \mathrm{dist}_H(E[r],C_0(E)[r])$ is definable. Then $f$ is continuous on some interval $(0,\varepsilon)$. Of course, $f(0)=0$. The Local Conical Structure Theorem ensures that $f$ is continuous at zero and either \cite{M} (in a polynomially bounded o-minimal structure) or the Puiseux expansion as presented in \cite{DS} gives the result.
\end{proof}
\begin{rem}
Theorems \ref{definiowalne} and \ref{Kuratowski} can be applied to $$E\ni x\mapsto \mathrm{dist}_K(C_x(E), C_a(E))$$ where $a\in E$ is fixed which gives a rate of convergence that can be useful for instance in Whitney stratifications (as there is a nice continuity property of tangents along the strata). 
\end{rem}

\section{\L ojasiewicz Inequality for subanalytic Lipschitz functions}\label{LojasiewiczSection}

All the previous results give the possibility of considering generalized gradient-type inequalities as in the following theorem. In view of the great and well-founded popularity this kind of inequalities have met lately with, there is no need to motivate this section more in details. Suffice it to say that our approach is to some extent similar to that of Bolte, Daniilidis and Lewis from \cite{BDLew} or Ph\d{a}m from \cite{Ph}. In both cases the authors consider, however, the {\it limiting subdifferential} of subanalytic continuous functions, whereas we shall concentrate on {\it Clarke's subgradient} of a subanalytic Lipschitz function as in the last section. The use of Clarke's subdifferential is unusual for it seldom gives the property sought for and we have thus an extra assumption: we are working at a point that is strongly critical in the sense that Clarke's subgradient reduces to zero.

The theorem we are aiming at is the following result that should be compared to \cite{BDLewS2} Corollary 16. Our approach is straightforward, but requires the extra assumption mentioned above.
\begin{thm}\label{subgradient}
   
   Let $f: \left( \mathbb{R}^n, 0 \right) \to \left(\mathbb{R}, 0\right)$ be subanalytic and Lipschitz continuous such that for the Clarke subgradient we have $\partial f(0) = \left\lbrace 0 \right\rbrace$.
   Define $h(x):= \inf \left\lbrace ||l|| \colon \ l \in \partial f(x) \right\rbrace$.
   
   Then there exist $U$ a neighbourhood of $0$ and constants $c > 0, \theta \in \left(0, 1 \right)$ such that:
   
   \begin{enumerate}
      \item $\forall x \in U, 0 \in \partial f(x) \implies f(x) = 0$,
      \item $h(x) \ge c  |f(x)| ^ {\theta}, \ x \in U$.
   \end{enumerate}
\end{thm}

 Before we prove this theorem we shall need some auxiliary results. But first let us consider one basic example (showing in particular that the set of differentiability points  $D_f$ need not be open).
\begin{ex}
Let $f(x,y)=|x|y$. This is a locally Lipschitz, semi-algebraic function whose non-differentiability points ${\Rz}^2\setminus D_f$ coincide with the $y$-axis without the origin. In particular we see that $D_f$ is not closed. For the subgradient we obtain
$$
\partial f(x,y)=\begin{cases}\{((\mathrm{sgn} x)y,|x|)\}, &\textrm{if}\ x\neq 0,\\
\{(0,0)\}, &\textrm{if}\ (x,y)=(0,0),\\
[-|y|,|y|]\times \{0\}, &\textrm{if}\ x=0, y\neq 0.  \end{cases}
$$
Note that $\partial f$ is continuous at zero in the sense of Kuratowski (actually, this is not fortuitous, cf. Lemma \ref{cglsc}). Here $D_f={\Rz}^2\setminus\{x=0, y\neq 0\}$. 

A straightforward computation shows that for this example the subgradient inequality (2) holds at zero with the best exponent $\theta=1/2$.
\end{ex}
Recall that a multifunction $F\colon {\Rz}^m\to\mathscr{P}({\Rz}^n)$ is \textit{outer semi-continuous} if at any point $x_0\in\mathrm{dom} F$ there is 
$$
\limsup_{x\to x_0} F(x)\subset F(x_0).
$$
As observed in \cite{DD}, this is automatically satisfied, if $\Gamma_F$ is closed. By \cite{Cl}, the subgradient is outer semi-continuous.

We need one more notion.
\begin{dfn}
The multifunction $F$ is said to be \textit{locally bounded}, if for any $x_0\in \mathrm{dom} F$ there is a neighbourhood $U$ of $x_0$ and  a constant $C>0$ such that for all $x\in U\cap\mathrm{dom} F$ and $y\in F(x)$, $||y||\leq C$. 
\end{dfn}
\begin{rem}
If in addition $\Gamma_F$ is closed, then this is equivalent to saying that the natural projection ${\Rz}^m\times{\Rz}^n\to {\Rz}^m$ is proper when restricted to $\Gamma_F$.
\end{rem}

\begin{lem}\label{cglsc}
Assume that $F\colon {\Rz}^m\to\mathscr{P}({\Rz}^m)$ is a locally bounded, outer semi-continuous, closed multifunction. Then for any point $x_0\in\mathrm{dom} F$ such that $\#F(x_0)=1$ we have
$$
F(x_0)=\lim_{x\to x_0} F(x)
$$
i.e. $x_0$ is a continuity point of $F$.
\end{lem}
\begin{proof}
We have only to check that $F(x_0)\subset \liminf_{x\to x_0} F(x)$ and we shall use Proposition \ref{obs}. Denote by $y_0$ the unique point of $F(x_0)$ and take any sequence $x_\nu\to x_0$ together with points $y_\nu\in F(x_\nu)$. Since $F$ is locally bounded, we may assume that the sequence $\{(x_\nu,y_\nu)\}$ is bounded. Now, from any subsequence $\{(x_{\nu_k},y_{\nu_k})\}$ we can extract a subsubsequence converging to some limit $(x(\nu_k),y(\nu_k))$ for which we necessarily have $x(\nu_k)=x_0$. This implies $y(\nu_k)\in \limsup_{x\to x_0} F(x)$ and so by assumptions, $y(\nu_k)=y_0$. We conclude that $y_\nu\to y_0$ which ends the proof.
\end{proof}
Of course, local boundedness in the lemma above is important:
\begin{ex}
Let $F$ be the function $F(x)=0$ for $x\geq 0$ and $F(x)=1/x$ for $x<0$. Then since $\Gamma_F$ is closed, $F$ is outer semi-continous as a multifunction, but it is not bounded near zero and indeed not continuous at this point.
\end{ex}

Let us note the following asymptotic gradient inequality in one variable:
\begin{lem}\label{gradAsympt}
If $\varphi\colon [0,1]\to{\Rz}$ is a continuous subanalytic function and $\varphi^{-1}(0)=\{0\}$, then there exist constants $c>0$ and $\theta\in (0,1)$ such that in some interval $(0,\varepsilon)$, $\varphi$ is $\mathscr{C}^1$ and
$$
|\varphi'(t)|\geq c|\varphi(t)|^\theta,\quad 0<t\ll 1.
$$
\end{lem}
\begin{proof}
Subanalycity implies that $\varphi$ is $\mathscr{C}^1$ on $(0,\varepsilon)$. Of course, from the Curve Selection Lemma we can also write $\varphi(t)=at^\alpha+o(t^\alpha)$ for $0\leq t\ll 1$. From the assumptions it follows that $c\neq 0$, $\alpha>0$. Actually, we are dealing here with the simplest case of a Puiseux expansion as noted in \cite{BochnakRisler}: for some integer $q>0$, $\varphi(t^q)$ has an analytic continuation through zero, hence near zero $\varphi$ is the sum of a convergent Puiseux series 
$$
\varphi(t)=\sum_{\nu\geq \nu_0} a_\nu t^{\nu/q},\quad 0\leq t\ll 1.
$$
with $a_{\nu_0}\neq 0$ and $\nu_0\geq 1$ (here $\alpha=\nu_0/q$). It follows easily that the derivative $\varphi'(t)$ (that exists at least for $0<t\ll 1$) admits an asymptotic Puiseux expansion obtained by differentiating the expansion of $\varphi$ term by term:
$$
\varphi'(t)=\sum_{\nu\geq \nu_0} a_\nu\frac{\nu}{q} t^{(\nu-q)/q},\quad 0< t\ll 1.
$$
In particular, only finitely many exponents $(\nu_0-q+k)/q$, $k=0,1,2,\dots$ are negative. It follows that $\varphi'(t)$ admits a limit $\ell=\lim_{t\to 0^+} \varphi'(t)$ that is either finite, or infinite. As $\lim_{t\to 0^+} \varphi(t)=\varphi(0)=0$ we see that the Lemma obviously holds true, if $\ell$ is finite and non-zero (this happens when $\alpha=1$), or infinite (when $\alpha\in (0,1)$) --- any $\theta$ works (\footnote{The point is that whenever $\alpha\in (0,1]$, we have actually $\varphi'(t)=c_1t^{-\beta_1}+\ldots+c_{k-1}t^{-\beta_{k-1}}+c_k t^\gamma+o(t^\delta)$ where $\beta_1>\ldots>\beta_{k-1}>\gamma\geq 0$ and $\delta>0$. It is clear that the limit when $t\to 0^+$ is $\mathrm{sgn} (c_{k-1})\infty$.}). The only case worth dealing with is when $\ell=0$ which means in particular that $\nu_0>q$ (i.e. $\alpha>1$). But then, if we try to compute for some $\theta>0$ the limit
$$
\lim_{t\to 0^+} \frac{|\varphi'(t)|}{|\varphi(t)|^\theta},
$$
 and obtain either $+\infty$ or a positive number, a simple computation shows that the largest exponent we may take is $\theta=(\nu_0-q)/\nu_0$. This $\theta$ belongs to $(0,1)$ which ends the proof.
\end{proof}

The following results of Bolte, Daniilidis, Lewis and Shiota are crucial in proving Theorem \ref{subgradient}.

\begin{thm}[\cite{BDLewS1}, Theorem 7]\label{connected}
Let $U$ be a nonempty subset of $\mathbb{R}^n$ and $f \colon U \to \mathbb{R}$ a locally Lipschitz subanalytic mapping. Let $S$ denote the set of Clarke critical points of $f$, that is, $$S: = \left\lbrace x \in U \colon \ \partial^\circ f(x) \ni 0 \right\rbrace,$$where $\partial^\circ f(x)$ is the Clarke subdifferential at $x$. Then $f$ is constant on each connected component of $S$. 
\end{thm}

%

\begin{rem}
In the paper \cite{BDLewS1} the authors use a slightly different definition of the Clarke subgradient calling it Clarke's subdifferential; there are two intermediary notions: that of the {\it Fr\'echet subdifferential} $\widehat{\partial} f$ and the {\it limiting subdifferential} denoted there $\partial f$ and that we will denote by $\partial'f$ (to disinguish it from the previously introduced notation for the Clarke subgradient). 
Namely, the Clarke subdifferential of $f$ at $x$, denoted by $\partial^\circ f(x)$, is defined there as the closed convex hull of the limiting sudifferential, i.e. of the set of points $x^{\star} \in \mathbb{R}^n$ for which there exist sequences $x_\nu \to x$ and $x_\nu^{\star} \to x^{\star}$ such that  $x_\nu^\star\in\widehat{\partial} f(x_\nu)$, which means that $$\liminf_{y \to x_\nu, y \neq x_\nu} \frac{f(y) - f(x) - \langle x_\nu^{\star}, y - x_\nu \rangle}{\Vert y - x_\nu \Vert} \ge 0.$$
According to Clarke's original definition, the subgradient is the convex hull of all points $x^*\in{\Rz}^n$ for which there exists a sequence $D_f\ni x_\nu\to x$ such that $\nabla f(x_\nu)\to x^*$. We do not even need to know that in fact both these definitions are equivalent, for in our special case of a subanalytic Lipschitz function germ $f$ we will just need the inclusion of the Clarke subgradient $\partial f(x)$ in the Clarke subdifferential $\partial^\circ f(x)$. 

Indeed, by \cite{Cl2} Theorem 2.5.1, in the definition of Clarke's subgradient we may throw away from $D_f$ any set of Lebesgue measure zero without affecting the result. As $f$ is subanalytic, it is $\mathscr{C}^1$ apart from a nowheredense subanalytic set $Z$, hence of measure zero. We may assume $Z$ to be closed. But then, when $x^*=\lim_{\nu\to+\infty} \nabla f(x_\nu)$ where $D_f\setminus Z\ni x_\nu\to x$, we know that $f$ is of class $\mathscr{C}^1$ around any point $x_\nu$. This implies (cf. \cite{BDLewS2} Remark 2) that the limiting and Fr\'echet subdifferentials at $x_\nu$ coincide and are equal to $\{\nabla f(x_\nu)\}$. Then by taking $x_\nu^\star:=\nabla f(x_\nu)$ we see that $x^*\in\partial' f(x)$ whence the inclusion sought for.

%
\end{rem}

Following \cite{BDLewS2}, a stratification $\lbrace \mathcal{S}_i \rbrace_{i \in I}$ of the graph $\Gamma_f$ that is a $\mathscr{C}^1$ stratification satisfying Whitney's property (a) (se e.g. \cite{DS} for these notions) is said to be \textit{nonvertical}, if $$\forall i \in I,\ \forall u \in \mathcal{S}_i, \ \ e_{n+1} = (0, 0, \dots , 1) \not \in T_u \mathcal{S}_i.$$

If $f$ is locally Lipschitz continuous, then it is easy to check that any Whitney $\mathscr{C}^1$ (a)-stratification of  $\Gamma_f$ is nonvertical. This is the case we will be dealing with in a neighbourhood of the origin in ${\Rz}^n$.

As observed in \cite{BDLewS2}, if we project the strata onto the domain of $f$, we obtain a Whitney stratification $\mathcal{X}=\{\mathcal{X}_i\}_{i\in I}$ of the domain and $f$ is $\mathscr{C}^1$ on each stratum $\mathcal{X}_i$. Then we denote by $\nabla_R f(x)$ the gradient (with respect to the inherited Riemann structure) of $f|_{\mathcal{X}_i}$ at $x$ where $\mathcal{X}_i$ is the stratum containing $x$. 

\begin{lem}[\cite{BDLewS2}, Corollary 5]\label{8.10}
Assume that $f\colon{\Rz}^n\to{\Rz}\cup\{+\infty\}$ is lower semi-continuous and admits a nonvertical Whitney stratification. Then $\forall x \in \mathrm{dom} \partial^\circ f$, we have $$ \Vert\nabla_R f (x) \Vert \le \Vert x^{\star} \Vert, \ \ x^{\star} \in  \partial^\circ f(x).$$
\end{lem}
\begin{rem}
We should note that under the assumptions of this Lemma the authors have to use again a slightly different definition of the Clarke subdifferential $\partial^\circ f$ as it might not be defined everywhere. If $f(x)=+\infty$, we assume it to be $\varnothing$, while for $x\in\mathrm{dom}f$, it is the closed convex hull of $\partial' f(x)+\partial^\infty f(x)$ where $\partial^\infty f$ denotes the {\it singular limiting subdifferential}. Fortunately, in our case: for a subanalytic Lipschitz function around the origin, this singular limiting subdifferential reduces to zero and we are left with the definition already mentioned in the last remark (cf. Section 4 in \cite{BDLewS1}).
\end{rem}

Now we are ready to prove our Theorem. 
\begin{proof}[Proof of Theorem \ref{subgradient}] 
The problem being local we may assume that $f$ is globally subanalytic and defined in an open ball $D$ centred at zero. Therefore, by Theorem \ref{definiowalne}, we are dealing with a subanalytic subgradient that in addition (cf. \cite{Cl}) is locally bounded, outer semi-continuous and has compact values. By Lemma \ref{cglsc}, it is continuous at zero. Of course, we assume $f\not\equiv 0$.

In order to prove (1) it suffices to show that $0$ does not belong to the closure of the subanalytic set 
$$
F:=\{x\in D\mid 0\in\partial f(x), f(x)\neq 0\}.
$$
Note that $0\notin F$. Suppose that $0\in\overline{F}$. Then by the Curve Selection Lemma there is an analytic curve $\gamma\colon ({\Rz},0)\to ({\Rz}^n,0)$ such that $\gamma((0,\varepsilon))\subset F$. Then $f(\gamma(t))\neq 0$ but $0\in\partial f(\gamma(t))$ for $t\in (0,\varepsilon)$.

Now, we see that $\gamma((0,\varepsilon))$ is contained in one connected component of the set $$S:=\{x\in D \mid 0\in\partial f(x)\}.$$ But by Theorem \ref{connected} this means that $f\circ \gamma$ is constant on $(0,\varepsilon)$. As there is $\lim_{t\to 0+} f(\gamma(t))=f(0)=0$, we conclude that $f\circ\gamma\equiv 0$ which is impossible, for $(f\circ\gamma)\neq 0$ on $(0,\varepsilon)$.

Now, we turn to proving (2). Note that $\partial f(x)$ is a compact set, therefore we may equivalently write $$h(x):= \min \left\lbrace ||l|| \colon \ l \in \partial f(x) \right\rbrace$$ and this is a subanalytic function due to the subanalycity both of the subgradient and the norm.

For $n = 1$, subanalycity implies that the function $f$ is $\mathscr{C}^1$  in a pointed interval $(-\varepsilon,\varepsilon)\setminus\{0\}$ and the assumption on the subgradient together with Lemma \ref{cglsc} implies that it is in fact $\mathscr{C}^1$ in the whole of $(-\varepsilon,\varepsilon)$. Therefore, in this case (2) reduces to the classical \L ojasiewicz gradient inequality and there is nothing to do.

Suppose $n > 1$. Fix a compact ball $B$ centered at zero, contained in the chosen domain $D$ of $f$. Notice that $h$ is lower semi-continuous --- it suffices to show that $\left\lbrace x\in D \mid \ h(x) > \alpha \right\rbrace$ is open for all $\alpha \in \mathbb{R}$.
Indeed, assume by contradiction that this set is not open. 
Then there exists $x_0$ such that $h(x_0) > \alpha, \ x_0 = \lim_{\nu \to \infty} x_\nu$ for some sequence such that $h(x_\nu) \le \alpha$. 
Also, for $\nu=0,1,\dots$, we have $h(x_\nu) = ||l_\nu||$ with some $l_\nu \in \partial f(x_\nu)$. 
By the outer semi-continuity of $\partial f$, $\limsup_{x \to x_0} \partial f (x) \subset \partial f (x_0)$ and these sets are compact.
 Thus, we can find a convergent subsequence $l_{\nu_k} \to l_0 \in \partial f (x_0)$ which is a contradiction, for $||l_{\nu_k}||\leq \alpha$, while $||l||>\alpha$.

For each $t\geq 0$ the set $\{x\in B\mid |f(x)|=t\}=|f|^{-1}(t) \cap B$ is compact. By the Darboux property, as we have assumed $f\not\equiv 0$, such a set is nonempty, provided $t\geq 0$ is small enough. Since a lower semi-continuous function attains its minimum on a compact set, $h$ attains its minimum on $\{x\in B\mid |f(x)|=t\}$. That is, we can define $$ \varphi: [0, \varepsilon) \ni t \mapsto \min \left\lbrace h(x) \mid \ x \in B, |f(x)| = t  \right\rbrace = \min E_t \in [0, \infty)$$ where $E_t = \left\lbrace u \in \mathbb{R} \colon (t,u) \in E \right\rbrace$ for the subanalytic set $$ E := (|f|, h) (B) \subset [0, +\infty) \times [0, \infty).$$ We see that $\varphi$ is a subanalytic function and $\varphi(0)=0$.
 
Notice that $\varphi(t) = 0$ iff there exists $x \in B \cap |f|^{-1}(t)$ such that $0 \in \partial f (x)$, but then, by (1), $f(x) = 0$, and so $t = 0$. That is, $\varphi(t)=0$ iff $t=0$.

As $\varphi$ is subanalytic, we may assume that $\varphi|_{(0, \varepsilon)}$ is $\mathscr{C}^1$ and monotonic. We now show that $\varphi$ is continuous at $0$.

To see this, let us fix $\eta>0$. 
By Lemma \ref{cglsc} and the assumptions, we know that $\partial f(x) \xrightarrow{K} \left\lbrace 0  \right\rbrace = \partial f (0)$, so we will find a closed ball $B_0$ centered at $0$ such that $\forall x \in B_0, \partial f(x) \subset B(0, \eta)$, hence $h(x) < \eta.$

By the continuity of $f$, we may choose a closed ball $B_1$ centered at $0$ such that $B_1 \subset B_0$ and there is an $\eta_0>0$ such that $|f|(B_1) = [0, \eta_0] \subset [0, \eta)$ (\footnote{Since $f\not\equiv 0$, $f(B_1)$ is a compact interval of the form $[0,\eta_0]$ with $\eta_1$ as small as we want, for an appropriate choice of $B_1$, because of $f(0)=0$.}).

Thus for any $t \in [0, \eta_0]$, we will find $x \in B_1$ such that $|f(x)| = t$ and 
$x\in B_1 \subset B_0$ implies $h(x) < \eta$. Consequently, $\varphi(t) < \eta$.

Thus we have proved that $\varphi \ge 0$ is continuous on $[0, \varepsilon)$ and
strictly increasing, as $\varphi$ vanishes only at zero and we know it is monotonic.

As a result, $\varphi$ has the following form: $$\varphi (t) = c t^{{\theta}} + o(t^{\theta}), \ t \in [0, \varepsilon), \textrm{for some}\ c>0, \ \theta >0.$$

It remains to be proven that $\theta \in (0, 1)$. For if it were so, then we would have $\Gamma _\varphi \cap V \subset \{(t,u)\mid u\geq t^{\tilde{\theta}} \}$ for some neighbourhood $V\ni 0$ and some $\tilde{\theta} \in (0,1)$ that we may take of the form  $\theta+\delta $ with $\delta \in (0,1-\theta)$ (\footnote{For then $\varphi (t)\geq t^{\tilde{\theta}}$ is equivalent to $1\leq \varphi (t)/t^{\tilde{\theta}}=(c+o(t^{\theta})/t^{\theta })/t^\delta $. The latter tends to $+\infty$ when $t\rightarrow 0^+$.}). This implies $(2)$.


Define $A: = \lbrace x \in B \setminus S \mid \ h(x) = \varphi(|f(x)|) \rbrace$. Clearly, this set is subanalytic. By (1),  $S \subset f^{-1}(0)$ and, as we have $f\not\equiv 0$, i.e. $0\notin\mathrm{int} f^{-1}(0)$, consequently, $0\in\overline{B \setminus f^{-1}(0)}$. Observe that $0\notin A$, since $0\in S$.

Note that $$x \in B \setminus S \iff 0 \not \in \partial f(x) \iff h(x) > 0.\leqno{(\circ)}$$

Recall that $\varphi(t) = \min \lbrace h(x) \mid \ x \in B \cap |f|^{-1}(t) \rbrace$. 
Namely, for any $t\in (0,\varepsilon)$, there exists $x_t \in B \cap |f|^{-1}(t)$ such that $\varphi(\vert f(x_t) \vert) = \varphi(t) = h(x_t)$ and $h(x_t)>0$, for $\varphi$ vanishes only at zero.
Therefore, $A \neq \varnothing$.

Take a sequence $$B \setminus f^{-1}(0) \ni x_\nu \to 0, $$ then $$0 < t_\nu: = |f(x_\nu)| \to 0.$$ 

We have that $$\forall \nu \ \exists \tilde{x}_\nu \in |f|^{-1}(t_\nu) \cap B \colon 0 < h(\tilde{x}_\nu) = \varphi(t_\nu) \to 0, \ (\nu \to \infty),$$thanks to the continuity of $\varphi$. 
By $(\circ)$, such $\tilde{x}_\nu \not \in S$.
Keeping in mind that $|f(\tilde{x}_\nu)| = t_\nu$, together with $h(\tilde{x}_\nu) = \varphi(t_\nu)$, we see that $\tilde{x}_\nu \in A$.

By assumption, $B$ is compact, so passing to a subsequence, if necessary, we find a limit $A \ni \tilde{x}_\nu \to x_0 \in B$. We have $|f(\tilde{x}_\nu)|=t_\nu\to 0$, so that $f(x_0)=0$. Actually, $x_0 \in S$. 
Indeed, we know that $h$ takes nonnegative values. If there were $h(x_0) >0$, then $h(x_0)>C>0$ and by the lower semi-continuity of $h$, we would have $h(x) > C$ in a neighborhood of $x_0$ which contradicts $h(\tilde{x}_\nu)\to 0$, $\tilde{x}_\nu\to x_0$. 

We have found a point $x_0 \in \overline{A} \setminus A$ and so the Curve Selection Lemma gives us an analytic arc $\gamma\colon \left( \mathbb{R}, 0 \right) \to \left( \mathbb{R}^n, x_0 \right)$ such that $\gamma ((0, r)) \subset A$. Notice that $f(\gamma(t))\to f(x_0)=0$ when $t\to 0^+$.
Then, \begin{gather}
h \left(\gamma(t)\right) = \varphi \left( \vert f \left( \gamma(t) \right) \vert \right) = c \vert f \left( \gamma(t) \right) \vert ^{\theta} + o \left( \vert f \left( \gamma (t) \right)\vert ^{\theta} \right),\quad 0<t\ll 1. \tag{*}
\end{gather}

In order to get $\theta < 1$, 
 it suffices to prove that \begin{gather}
\lim_{t \to 0^+} \frac{\vert f \left( \gamma(t) \right) \vert}{h(\gamma(t))} = 0. \tag{**}
\end{gather}
Indeed, we have that $h \left( \gamma (t) \right) \to 0, \ \vert f \left( \gamma(t) \right) \vert \to 0$ as $t \to 0^+$. 
Dividing both sides of $(\hbox{}^*)$ by $h \left(\gamma(t) \right)^{\theta}$ gives:
$$h \left(\gamma(t) \right)^{1- \theta} = c \cdot \left(\frac{\vert f \left( \gamma(t) \right) \vert }{h(\gamma(t))} \right)^{\theta} + \frac{o \left( \vert f \left( \gamma (t) \right)\vert ^{\theta} \right)}{ \vert f \left( \gamma (t) \right)\vert ^{\theta}} \cdot \frac{ \vert f \left( \gamma (t) \right)\vert ^{\theta}}{h \left( \gamma (t) \right)^{\theta}}.$$
Passing to the limit when $t\to 0^+$ shows that there must be $ 1 - \theta > 0$.

Let us now prove $(\hbox{}^{**})$. Assume by contradiction that $$\exists \tilde{c}>0, \ \exists  t_\nu \to 0^+ \colon \ \frac{\vert f \left( \gamma(t_\nu) \right) \vert}{h(\gamma(t_\nu))} \ge \tilde{c}.$$
Then for the subanalytic set $$E:= \lbrace t \in (0,r) \mid \ \vert f \left( \gamma(t) \right) \vert \ge \tilde{c}\cdot  h \left( \gamma(t) \right) \rbrace \neq \varnothing,$$ we have
$0 \in \overline{E} \setminus E$. The subanalycity of $E$ implies that there exists $r_1>0$ such that $ (0, r_1) \subset E$ and so $$ \vert f \left( \gamma(t) \right) \vert \ge \tilde{c} \cdot h \left( \gamma(t) \right) > 0, \ t \in (0, r_1). $$

Define $\psi(t): = (f \circ \gamma)(t), \ t \in [0, r_1)$. Thus defined $\psi$ is continuous subanalytic on $[0, r_1)$ and $$\psi(t) = 0 \iff t = 0.$$
Therefore $\psi$ has the following representation near zero: $$\psi(t) = a t^{\alpha} + o(t^{\alpha}), \ a \neq 0, \ \alpha > 0.$$

Now, by Lemma \ref{gradAsympt}, $\psi$ satisfies the asymptotic \L ojasiewicz gradient inequality, regardless of $\psi ^{\prime} (0) = 0$:
$$\vert \psi^{\prime} (t) \vert \ge \delta\vert \psi (t) \vert ^{\beta}, \ 0< t\ll 1,$$
for a some constants $\delta>0$ and $\beta \in (0,1)$.

Take now a nonvertical stratification of the graph of $f$ over $B$ and let $\left\lbrace \mathcal{X}_i \right\rbrace_{i \in I}$ be the Whitney stratification of $B$ obtained from it by projection. 
From the analycity of $\gamma$ it follows that there is exactly one $i$ such that for some $0<r_2\leq r_1$, we have $\gamma((0,r_2))\subset\mathcal{X}_i$. Then $x_0\in\overline{\mathcal{X}_i}$, $\psi=f\circ\gamma$ is of class $\mathscr{C}^1$ on $(0,r_2)$ and $\gamma'(t)$ lies in the tangent space $T_{\gamma(t)}\mathcal{X}_i$, for any $0<t<r_2$. From this we infer that $\psi'(t)=\langle \nabla_R f(\gamma(t)),\gamma'(t)\rangle$ (actually, the same kind of observation is made in \cite{Drus} Lemma 2.10 as a consequence of the methods used in \cite{BDLewS2}). But the derivative $\gamma'(t)$ is bounded by some $M>0$ and so 
$$
\delta |\psi(t)|^\beta\leq M||\nabla_R f(\gamma(t))||,\quad 0<t\ll 1.
$$
Using Lemma \ref{8.10}, we obtain eventually
$$
\delta|f(\gamma(t))|^\beta\leq M h(\gamma(t)),\quad 0<t\ll 1.
$$
However, by the choice of $\gamma$ we have also that for all $t$ small enough,
$$
0<\tilde{c}\cdot h(\gamma(t))\leq |f(\gamma(t))|.
$$
Summing up, we have found an exponent $\beta\in (0,1)$ such that for all $t$ sufficiently small,
$$
0<\frac{\delta\tilde{c}^\beta}{M}\leq h(\gamma(t))^{1-\beta},
$$
where the right-hand side converges to zero, as $t\to 0^+$. This contradiction ends the proof of the theorem.

\end{proof}

\section{Acknowledgements}

During the preparation of this paper, the first-named author was supported by Polish Ministry of Science and Higher Education grant  IP2011 009571.

He also would like to thank the University Lille 1, where the first part of this work was done, for hospitality and professor T.S. Ph\d{a}m for an interesting discussion on \cite{Ph} and the references to \cite{Drus} and \cite{BDLewS2}.


\end{document}